\documentclass{article}
\usepackage{amsmath,amssymb,amsthm}
\usepackage[pagebackref]{hyperref}

\newcommand{\Q}{{\mathbb Q}}
\newcommand{\R}{{\mathbb R}}
\newcommand{\N}{{\mathbb N}}
\newcommand{\Z}{{\mathbb Z}}
\newcommand{\OO}{\mathcal O}
\newcommand{\fp}{\mathfrak p}
\newcommand{\fq}{\mathfrak q}

\newcommand{\K}{\overline{K}}
\newcommand{\bM}{\overline{M}}
\newcommand{\Ft}{\widetilde{F}}
\newcommand{\disc}{\operatorname{disc}}

\newcommand{\Orb}{\operatorname{Orb}}

\newcommand{\eps}{\varepsilon}
\newcommand{\oeps}{\overline{\varepsilon}}

\newcommand{\lra}{\longrightarrow}
\newcommand{\too}{\longmapsto}
\newcommand{\ts}{\textstyle}
\newcommand{\subs}{\widetilde{\subset}}
\def\rsp{\raisebox{0em}[2.4ex][1.2ex]{\rule{0em}{2ex}}}

\newtheorem{thm}{Theorem}[section]
\newtheorem{prop}[thm]{Proposition}
\newtheorem{lem}[thm]{Lemma}

\theoremstyle{definition}

\newtheorem{claim}{Claim}[section]

\begin{document}
\title{Euclidean Windows}
\author{
   Stefania Cavallar 
 \and 
   Franz Lemmermeyer }
\maketitle

\abstract
In this paper we study number fields which are Euclidean
with respect to a function different from the absolute
value of the norm. We also show that the Euclidean minimum
with respect to weighted norms may be irrational and not 
isolated.
\endabstract

\section*{Introduction}

Let $R$ be an integral domain. A function $f:R \lra \R_{\ge 0}$ 
is called a Euclidean function on $R$ if it satisfies the following 
conditions with $\kappa = 1$:
\begin{enumerate}
\item[i)]  $f(R) \cap [0,c]$ is finite for every $c \ge 0$;
\item[ii)] $f(r) = 0$ if and only if $r = 0$;
\item[iii)] for all $a, b \in R$ with $b \ne 0$ there exists a 
      $q \in R$ such that $f(a-bq) < \kappa \cdot f(b)$.
\end{enumerate}
If $f:R \lra \R_{\ge 0}$ is a function satisfying i) and ii), 
then the infimum of all $\kappa \in \R$ such that iii) holds 
is called the Euclidean minimum of $R$ with respect to $f$ and 
will be denoted by $M(R,f)$; thus for all $a, b \in R \setminus \{0\}$
and every $\eps > 0$ there is a $q \in R$ such that 
$f(a-bq) < M(R,f)\cdot f(b) + \eps$. 

If $f$ is a multiplicative function, then we can replace iii)
by the equivalent condition that for all $\xi \in K$ ($K$
being the quotient field of $R$) there is a $q \in R$
such that $f(\xi - q) < \kappa$. The infimum of all $\kappa \in \R$ 
such that this condition holds for a fixed $\xi$ is denoted
by $M(\xi,f)$; clearly $M(R,f)$ is the supremum of the $M(\xi,f)$.

If $R = \OO_K$ is the ring of integers in a number field $K$,
then the absolute value of the norm satisfies i) and ii), and 
$M(K) := M(R,|N|)$ coincides conjecturally with the inhomogeneous 
minimum of the norm form of $\OO_K$ (this conjecture is known to 
hold for number fields with unit rank at most $1$). Let 
$C_1$ be the set of representatives modulo $\OO_K$ of all 
$\xi = \frac{a}{b} \in K$ with $M(\xi)  = M(K)$ (here 
$M(\xi) := M(\xi,|N|)$); then we say that $M(K)$ is isolated 
if there is a $\kappa_2 < \kappa$ such that $M(\xi) \le \kappa_2$ 
for all $\xi \in K$ that are not represented by some point in $C_1$. 

Replacing $K$ in these definitions by $\K = \R^n$ (this 
is the topological closure of the image of $K$ und the standard
embedding $K \lra \R^n$; for totally real fields we have
$\K = K \otimes_\Q \R$), the Euclidean minimum becomes the 
inhomogeneous minimum of the norm form of $K$; we clearly
have $M_j(\K) \ge M_j(K)$ whenever these minima are defined,
and it is conjectured that $M_1(\K) = M_1(K)$ is rational.

The aim of this paper is to explain how the Euclidean minimum
of $\OO_K$ with respect to ``weighted norms'' can be computed
in some cases; we will show that the Euclidean minimum
for certain weighted norms in $\Q(\sqrt{69}\,)$ is irrational
and not isolated, thereby showing that these conjectured
properties for minima with respect to the usual norm do not
carry over to weighted norms.

\section{Weighted norms}

Let $K$ be a number field, $\OO_K$ its ring of integers,
and $\fp$ a prime ideal in $\OO_K$. Then, for
any real number $c > 0$, 
$$\phi: \fq \too \begin{cases} N\fq, & \text{ if } \fq \ne \fp \\
                        c,   & \text{ if } \fq = \fp \end{cases} $$
defines a map from the set of prime ideals $\fq$ of $\OO_K$ into 
the positive real numbers, which can be uniquely extended to a 
multiplicative map $\phi:I_K \lra \R_{>0}$ on the group $I_K$ of 
fractional ideals.
Putting $f(\alpha) = \phi(\alpha\OO_K)$ for any $\alpha \in K^\times$
and $f(0) = 0$, we get a function $f = f_{\fp,c}: K \lra \R_{\ge 0}$
which H.~W.~Lenstra \cite{Len} called a {\em weighted norm}.

Our aim is to study examples of number fields which are Euclidean
with respect to some weighted norm. Lenstra \cite{Len} showed
that $\Q(\zeta_3)$ and $\Q(\zeta_4)$ are such fields, but the 
first examples that are not norm-Euclidean were given by D.~Clark 
\cite{Cl1,Cl2}.

A formal condition for $f_{\fp,c}$ to be a Euclidean
function is the finiteness of the sets 
$\{f_{\fp,c}(\alpha) < \lambda : \alpha \in \OO_K \}$
for  all $\lambda \in \R$. This property is easily seen to be
equivalent to $c > 1$.


For weighted norms $f = f_{\fp,c}$ on $K$, 
we define the {\em Euclidean window} of $\fp$ by
$$ w(\fp) = \{ c \in \R: f_{\fp,c} 
\text{ is a Euclidean function on } \OO_K \}.$$

\begin{prop}
The Euclidean window is a (possibly empty) interval contained
in $(1,\infty)$.
\end{prop}

\begin{proof}
Assume that $w(\fp)$ is not empty, and let $r,t \in w(\fp)$ with
$r < t$. Then it is sufficient to show that $f_{\fp,s}$ is a 
Euclidean function on $\OO_K$ for every $r \le s \le t$. Now
$\OO_K$ is Euclidean with respect to e.g. $f_{\fp,r}$, so $\OO_K$ 
is a principal ideal domain, hence every $\xi \in K$ has the 
form $\xi = \alpha/\beta$ with $(\alpha,\beta) = 1$. Moreover, 
there exist $\gamma_r, \gamma_t \in \OO_K$ such that 
$$  f_{\fp,r}(\alpha - \beta\gamma_r) < f_{\fp,r}(\beta), \qquad
    f_{\fp,t}(\alpha - \beta\gamma_t) < f_{\fp,t}(\beta).$$
If $\fp \nmid \beta$, then $ f_{\fp,s}(\alpha - \beta\gamma_t) \le
 f_{\fp,t}(\alpha - \beta\gamma_t) < f_{\fp,t}(\beta)
  = f_{\fp,s}(\beta)$; 
if $\fp \mid \beta$, on the other hand, then $\fp \nmid \alpha$, 
hence $\fp \nmid (\alpha - \beta\gamma_r)$, and
$f_{\fp,s}(\alpha - \beta\gamma_r) = f_{\fp,r}(\alpha - \beta\gamma_r)
  < f_{\fp,r}(\beta) \le  f_{\fp,s}(\beta).$
Thus $f_{\fp,s}$ is indeed a Euclidean function on $\OO_K$.
\end{proof}

In this paper, we investigate Euclidean windows for various
algorithms in some quadratic and cubic number fields; we will
give examples of empty, finite and infinite Euclidean windows,
and we show that the first minima with respect to weighted
norms need not be rational.

\section{Weighted norms in $\Z$}

The Euclidean window for primes in $\Z$ can easily
be determined:

\begin{prop}
The Euclidean minimum $M(f_{p,c})$ of a weighted norm in $\Z$
is given by
$$ M(f_{p,c}) = \begin{cases}  \infty & \text{if } c < p \\
                        \frac12 & \text{if } c = p \\
                        1  & \text{if } c > p \end{cases}.$$
Moreover, $w(p) = [p,\infty)$.
\end{prop}

\begin{proof}
We first show that $M(f_{p,c}) = \infty$ if $c < p$ (this implies
that $w(\fp) \subseteq [p,\infty)$). To this end, put $b = p^n$ and 
$$ a = \begin{cases} \frac12(p^n-1) & \text{if } p \ne 2,\\
                2^{n-1}-1    & \text{if } p  = 2. \end{cases}$$
Then $p \nmid (a-bq)$, hence $f_{p,c}(a-bq) = |a-bq|$ for all 
$q \in \Z$. If the minimum $\kappa = M(f_{p,c})$ were finite,
there would exist a $q \in \Z$ such that 
$f_{p,c}(a-bq) < \kappa f_{p,c}(b) = \kappa c^n$. But clearly
$|a| \le |a-bq| = f_{p,c}(a-bq)$, hence we get
$|a|c^{-n} < \kappa$ for all $n \in \N$: but since $c < p$, the 
expression on the left hand side tends to $\infty$ with $n$. 

Since it is well known that $M(f_{p,p}) = \frac12$, we next show 
that $M(f_{p,c}) = 1$ if $c > p$. To this end, choose 
$\alpha, \beta \in \N$ not divisible by $p$ such that
$p < \frac{\alpha}{\beta} < c$. If we put $a = p^n\beta^n$ and
$b = \alpha^n + p^n\beta^n$, then we get 
\begin{eqnarray*}
f_{p,c}\Big(\frac{a}{b}\Big) 
        & = & \frac{c^n \beta^n}{\alpha^n +  p^n\beta^n} 
        = \frac{c^n}{(\alpha/\beta)^n + p^n} 
        > \frac{c^n}{c^n+ p^n}, \\
f_{p,c}\Big(\frac{a}{b} - 1 \Big) 
        & = & \frac{\alpha^n}{\alpha^n + p^n\beta^n},
\end{eqnarray*}
and both expressions tend to $1$ as $n$ goes to $\infty$. 
Note also that $f_{p,c}(\frac{a}b-q) \ge |\frac{a}b-q| > 1$ for all
$q \in \N \setminus \{0, 1\}$, since the denominator of
$\frac{a}b-q$ is prime to $p$ and since $c > p$. 

Thus $M(f_{p,c}) \ge 1$ if $c > p$; but we can easily show
that $M(f_{p,c}) \le 1$ by proving that $f_{p,c}$ is a 
Euclidean function for all $p \ge c$. In fact,
suppose that $a, b \in \Z\setminus \{0\}$ are given,
and that they are relatively prime. If $p \mid b$,
then $p \nmid (a-bq)$ for all $q \in \Z$, hence 
$f_{p,c}(a-bq) = |a-bq|$, and we can certainly 
find $q \in \Z$ such that $|a-bq| < |b|$. But 
$|b| \le f_{p,c}(b)$ since $c \ge p$. 

Now consider the case $p \nmid b$; then we choose $q \in \Z$
such that $|a-bq|, |a-b(q+1)| \le b$. But $r = a-bq$ and
$r' = a-b(q+1)$ cannot both be divisible by $p$; if $p \nmid r$,
then $f_{p,c}(r) = |r| < |b| = f_{p,c}(b)$, and if $p \nmid r'$,
then $f_{p,c}(r') < f_{p,c}(b)$. 
\end{proof}

\section{Weighted norms in $\Q(\sqrt{14}\,)$}

Since it is well known that an imaginary quadratic
number field is Euclidean if and only if it is norm-Euclidean,
only the case of real quadratic fields is interesting. 
We will deal with only two examples here: one is
$\Q(\sqrt{14}\,)$, which has been studied often in this
respect (cf. Bedocchi \cite{Bed}, Nagata \cite{Na1,Na2} and
Cardon \cite{Car}), and the other is $\Q(\sqrt{69}\,)$, 
which was shown to be Euclidean with respect to a weighted 
norm by Clark \cite{Cl1} (see also Niklasch \cite{Nik} and
Hainke \cite{Hai}). 

Consider the quadratic number field $K = \Q(\sqrt{14}\,)$. 
It is well known that $M_1(K) = \frac54$
and $M_2(K) = \frac{31}{32}$ (cf. \cite{Lem}); moreover
$M_1$ is attained exactly at the points 
$\xi \equiv \frac12(1+\sqrt{14}\,) \bmod \OO_K$. Now we claim

\begin{prop}
For $K = \Q(\sqrt{14}\,)$ and $\fp = (2,\sqrt{14}\,)$ we have
$w(\fp) \subseteq (\sqrt5, \sqrt7)$.
\end{prop}

\begin{proof}
Put $\alpha = 1+\sqrt{14}$, $\beta = 2$. Then
$|N(\alpha - \beta\gamma)|$ is an odd integer $\ge 5$ for all 
$\gamma \in \OO_K$. Thus $f_{\fp,c}(\alpha - \beta\gamma)
= |N(\alpha - \beta\gamma)| \ge 5$, and if $f_{\fp,c}$ is
a Euclidean function, we must have $5 < f_{\fp,c}(\beta) = c^2$.
This shows that $c > \sqrt{5}$.

In order to show that $c < \sqrt{7}$ we look at the ideal
$\fq = (7,\sqrt{14}\,) = (7+2\sqrt{14}\,)$ of norm $7$. 
If $f_{\fp,c}$ is Euclidean, then every residue class modulo
$\fq$ must contain an element $\alpha$ such that 
$f_{\fp,c}(\alpha) < f_{\fp,c}(\fq) = 7$. Since the unit
group generates the subgroup $\{-1, +1\}$ of $(\OO_K/\fq)^\times$
(and $f_{\fp,c}(\pm 1) = 1$), and since 
$\pm 3+\sqrt{14} \equiv \pm 3 \bmod \fq$ (where 
$f_{\fp,c}(\pm 3+\sqrt{14}) = |N(\pm 3+\sqrt{14})| = 5$),
we must find elements in the residue classes $\pm 2 \bmod \fq$.
The only possible candidates are powers of $4+\sqrt{14}$, because
the only ideals of odd norm $< 7$ are $(0)$, $(1)$, and 
$(3 \pm \sqrt{14}\,)$, none of which yields elements 
$\equiv \pm 2 \bmod \fq$. Moreover, $\pm 4+\sqrt{14} \equiv 
\pm 3 \bmod \fq$, and we see that if there exist elements
$\alpha \equiv 2 \bmod \fq$ with $f_{\fp,c}(\alpha) < 7$,
then $\alpha = 2$ is one of them. But $f_{\fp,c}(2) = c^2$,
and we find $c < \sqrt{7}$. 
\end{proof}

We remark that it is not known whether  $w(\fp)$ is empty or not.

If we look at prime ideals other than $(2,\sqrt{14}\,)$,
the situation is quite different: 

\begin{prop}
Let $K = \Q(\sqrt{14}\,)$, and let $\fp$ be a prime ideal
in $\OO_K$ of norm $N\fp \equiv \pm 1 \bmod 8$. 
Then $w(\fp) = \varnothing$.
\end{prop}

\begin{proof}
Assume that $f_{\fp,c}$ is a Euclidean function. Then there
exists an $\alpha = x+y\sqrt{14} \equiv 1 + \sqrt{14} \bmod 2$
such that $f_{\fp,c}(\alpha) < f_{\fp,c}(2) = 4$. 
Since $\alpha$ cannot be a unit, this is only possible if 
$\alpha$ is divisible by $\fp$. If $\alpha$ is divisible by 
some other prime ideal $\fq$, then $f_{\fp,c}(\fq) = N\fq \ge 5$,
and we conclude $f_{\fp,c}(\fp) < 1$: contradiction.
Thus $(\alpha) = \fp^m$ for some $m \ge 1$. But 
$\fp = (a+b\sqrt{14}\,)$ since $K$ has class number $1$,
and $b$ must be even since $\pm p = a^2 - 14b^2 \equiv \pm 1 \bmod 8$:
thus $a+b\sqrt{14} \not\equiv 1 + \sqrt{14} \bmod 2$, and again
we have a contradiction.
\end{proof}

\section{The Euclidean Algorithm in $\Q(\sqrt{69}\,)$}

Next we study the field $\Q(\sqrt{69}\,)$; we will prove the
following result that corrects a claim\footnote{namely that
$M_2(K) < M_2(\K)$, and that $M_2(\K)$ is isolated.} 
announced without proof in \cite{Lem}:

\begin{thm}\label{T69}
 In $K = \Q(\sqrt{69}\,)$, we have                       
   $$ \begin{array}{ll} \smallskip
        M_1 = \frac{25}{23}, 
      & C_1  =  \big\{\pm \frac4{23}\sqrt{69}\big\}, \\
        M_2  = \frac1{46}\left(165-15\sqrt{69}\,\right), 
      & C_2  = \big\{(\pm P_r, \pm P_r')\big\}, r \ge 0 
\end{array} $$ 
where 
$$ P_r  = \frac12 \varepsilon^{-r} + \left(\frac4{23} + 
           \frac1{2\sqrt{69}}\varepsilon^{-r}\right)\sqrt{69}, \quad
          P_r'  = \frac12 \varepsilon^{-r} - \left(\frac4{23} + 
           \frac1{2\sqrt{69}}\varepsilon^{-r}\right)\sqrt{69}.$$
Here $M_j$ denotes the $j$-th inhomogeneous minimum of the
norm form of $\OO_K$, $C_j$ is a set of representatives
modulo $\OO_K$ of the points where $M_j$ is attained, and 
$\eps = \frac12(25+3\sqrt{69}\,)$ is the fundamental unit of $K$.
The second minimum $M_2(K) = M_2(\K)$ is not isolated.
\end{thm}

The proof of Theorem \ref{T69} is based on methods developed
by Barnes and Swinnerton-Dyer \cite{BSD}. 
In the following, we will regard $K$ as a subset of $\R^2$ via 
the embedding $x+y\sqrt{69}) \lra (x,y)$. Conversely, any point
$P = (x,y) \in \R^2= \K$ corresponds to a pair 
$\xi_P = x + y\sqrt{69}$, $\xi_P' = x-y\sqrt{69}$. These
elements are not necessarily in $K$; nevertheless we call 
$\xi_P' = x-y\sqrt{69}$ the conjugate of $\xi_P$. Note that
e.g. $\xi_P = \sqrt{69}$ alone does not determine $P$, since both
$P = (0,1)$ and $P = (\sqrt{69},0)$ correspond to such a $\xi_P$.
The ``$\K$-valuations'' $|\,\cdot\,|_1$ and $|\,\cdot\,|_2$ 
are defined by $|(x,y)|_1 = | x + y\sqrt{69}|$ and 
$|(x,y)|_2 = | x - y\sqrt{69}|$, with a positive square root of $69$.

Using the technique described in \cite{CL}, it is easy to cover
the whole fundamental domain of the lattice $\OO_K$ with
a bound of $k = 0.875$ except
for $\pm S_0 \cup \pm S_1 \cup \pm S_2 \cup \pm T$, where 
$$ \begin{array}{rcrcl}
  S_0 & = & [-0.00085, \phantom{-}0.00085] & \times & [0.1739, 0.1742] \\
  S_1 & = & [\phantom{-}0.01917,  \phantom{-}0.02005] & 
              \times & [0.1763, 0.1765] \\
  S_2 & = & [-0.02005, -0.01917] & \times & [0.1763, 0.1765] \\
  T   & = & [\phantom{-}0.4999\phantom{0},  \phantom{-}0.5001\phantom{0}] & 
               \times & [0.2341, 0.2342]. \\ 
\end{array} $$
Transforming these exceptional sets by multiplication with
the units $\eps$ and $\oeps = \frac12(25-3\sqrt{69}\,)$ we find e.g.
$$ \eps S_0 \subset 18 + 2\sqrt{69} + [-0.012, 0.041] \times 
[0.172, 0.179],$$
that is, $\eps S_0 - (18+2\sqrt{69}\,)$ is contained in covered
regions or $S_0 \cup S_1$, which we will denote by 
$\eps S_0 - (18+2\sqrt{69}\,) \subs S_0 \cup S_1$. 
 Similar calculations show that
$$ \begin{array}{rclrcl}
\eps S_0 - (18+2\sqrt{69}\,) & \subs & S_0 \cup S_1, & 
     \oeps S_0 + (18 - 2\sqrt{69}\,) & \subs & S_0 \cup S_2, \\
\eps S_1 - (18+2\sqrt{69}\,) & \subs & T, & 
     \oeps S_1 + (18 - 2\sqrt{69}\,) & \subs & S_0 \cup S_2, \\
\eps S_2 - (18+2\sqrt{69}\,) & \subs & S_0 \cup S_1, &
     \oeps S_2 + (19 - 2\sqrt{69}\,) & \subs & T,  \\
\eps T   - \frac12(61 + 7 \sqrt{69}\,) & \subs & S_2, & 
     \oeps \,T   + (18 - 2\sqrt{69}\,) & \subs & S_1.
\end{array}$$

\noindent{\bf Remark.} The inclusions on the right hand side can
be computed from those on the left: for example, all 
exceptional points in $S_2$ must come from $T$, so the exceptional
points in $\eps^{-1}S_2$ must be congruent modulo $\OO_K$ to points
in $T$, and since $\frac12(61 + 7\sqrt{69}\,) \eps^{-1}
= 19-2\sqrt{69}$, we conclude that $\oeps S_2 + (19 - 2\sqrt{69}\,) 
\subs T$.

\medskip

We will need the following result (this is Prop. 2 of \cite{CL}):
\begin{prop}\label{PEP}
Let $K$ be a number field and $\eps$ a non-torsion unit of $E_K$. 
Suppose that $S \subset \Ft$ has the following property:
\begin{quotation}
There exists a unique $\theta \in \OO_K$ such that, 
for all $\xi \in S$, the element $\eps\xi - \theta$ lies 
in a $k$-covered region of $\Ft$ or again in $S$.
\end{quotation}
Then every $k$-exceptional point $\xi_0 \in S$ satisfies
$|\xi_0-\frac{\theta}{\eps-1}|_j = 0$ for every $\K$-valuation
$|\cdot|_j$ such that $|\eps|_j > 1$.
\end{prop}

We also need a method to compute Euclidean minima of 
given points. Recall that the orbit of $\xi \in \K$
is the set $\Orb(\xi) = \{\eps\xi: \eps \in E_K\}$, 
where $E_K$ is the unit group of $\OO_K$. Note that 
all the elements in an orbit have the same minimum.

\begin{prop}\label{Pbd}
Let $m \in \N$ be squarefree, $K = \Q(\sqrt{m}\,)$ a real quadratic
number field, $\eps > 1$ a unit in $\OO_K$, and $\xi \in \K$.
If $M(K,\xi) < k$ for some real $k$, then there exists an element
$\eta = r+s\sqrt{m} \in K$ with the following properties:
\begin{enumerate}
\item[i)] $\eta \equiv \xi_j \bmod \OO_K$ for some $\xi_j \in \Orb(\xi)$;
\item[ii)] $|N\eta | < k$;
\item[iii)] $|r| < \mu$, $|s| < \frac{\mu}{\sqrt{m}}$, where 
      $\mu = \frac{\sqrt{k}\,}2\Big(\sqrt{\eps} + \frac1{\sqrt{\eps}}\Big)$.
\end{enumerate}
\end{prop}

\begin{proof}
Assume that $M(K,\xi) < k$; then there is an $\alpha \in \OO_K$
such that $|N(\xi-\alpha)| < k$. Choose $m \in \Z$ such that
$\sqrt{k/\eps} \le |(\xi-\alpha)\eps^m| < \sqrt{k\eps}$ and
put $\eta = (\xi-\alpha)\eps^m$. Then
\begin{enumerate}
\item[i)] $\eta = (\xi-\alpha)\eps^m \equiv \xi\eps^m \bmod \OO_K$,
     and clearly $\xi\eps^m \in \Orb(\xi)$;
\item[ii)] $|N\eta | = |N(\xi-\alpha)| < k$;
\item[iii)] Write $\eta = r+s\sqrt{m}$ and $\eta' = r-s\sqrt{m}$.
     Then $|\eta| < \sqrt{k\eps}$ and  
     $|\eta'| = |\eta\eta'|/|\eta| < k/|\eta| \le \sqrt{k\eps}$. 
     Thus $2|r| = |\eta + \eta'| \le |\eta| + |\eta'|$ and
     $2|s|\sqrt{m} =  |\eta - \eta'| \le |\eta| + |\eta'|$.
     Using the lemma below, this yields the desired bounds.
\end{enumerate}
This concludes the proof. \end{proof}

\begin{lem}
If $x, y$ are positive real numbers such that $x<a$, $y<a$ and $xy<b$,
then $x+y < a+\frac{b}a$.
\end{lem}

\begin{proof}
$0 < (a-x)(a-y) = a^2 - a(x+y) + xy < a^2 - a(x+y) + b$.
\end{proof}

\noindent Now we are ready to determine a certain class of exceptional
points inside $S_0$:

\begin{claim}
If $P$ is an exceptional point in $S_0$ that stays inside 
$S_0$ under repeated applications of the maps 
\begin{eqnarray} \label{Eal} 
    \alpha: & \xi \too \eps^{-1}\xi + 18 - 2\sqrt{69} \\
\label{Ebe} \beta: & \xi \too \eps \xi - (18 + 2\sqrt{69}\,)
\end{eqnarray}
then $P = \frac{18+2\sqrt{69}\,}{\eps-1} = (0,\frac{4}{23})$.
Moreover, $M(P) = \frac{25}{23}$.
\end{claim}

This follows directly from Proposition \ref{PEP};
the Euclidean minimum $M(P) = \frac{25}{23}$ is easily computed 
using Proposition \ref{Pbd}. Any exceptional point that does not 
stay inside $S_0$ must eventually come through $T$; it is 
therefore sufficient to consider exceptional points in $T$
from now on.

Let $P_0 \in T$ be such an exceptional point and define
the series of points $P_0$, $P_1$, $P_2$, \ldots 
recursively by $P_{j+1} = \alpha(P_j)$.  Then $P_1 \in S_1$,
and now there are two possibilities:
\begin{enumerate}
\item[(A)] $P_j \in S_0$ for all $j \ge 2$;
\item[(B)] there is an $n \ge 2$ such that $P_n \in S_1$.
\end{enumerate}

Before we can go in the other direction we have to adjust $P_0$
somewhat. In fact, $\beta(P_0) \in T$ implies that
$\beta(P_0) - \eps \subs S_2$; thus we can define  a
sequence of points $P_0-1$, $P_{-1}$, $P_{-2}$, $\ldots$ by 
$P_{-1} = \beta(P_0-1)$ and $P_{-j-1} = \beta(P_{-j})$ for $j \ge 1$.
Again, there are two possibilities:
\begin{enumerate}
\item[(C)] $P_{-j} \in S_0$ for all $j \ge 2$;
\item[(D)] there is an $n \ge 2$ such that $P_{-n} \in S_1$.
\end{enumerate}

\begin{claim}
If $P_0 \in T$ is an exceptional point satisfying conditions 
(A) and (C), then $P_0 = (\frac12, \frac4{23} + \frac1{2\sqrt{69}})
\approx (0.5, 0.234105)$.
\end{claim} 
Note that this point is not contained in $K$. Of course we knew 
this before: every point in $K$ has a finite orbit, whereas $P_0$
does not.

For a proof, we apply Proposition \ref{PEP} to the set 
$S = \{P_0, P_1, P_2, \ldots\}$; this shows that any $\xi = P_j$ 
lies on the line $|\xi + \frac{18-2\sqrt{69}}{\oeps - 1}|_2 = 0$ 
(the $\K$-valuation $|\cdot|_2$ chosen so that $|\oeps|_2 > 1$), 
that is, $\xi' = -\frac4{23}\sqrt{69}$. Applying the same proposition 
to $S = \{P_0-1, P_{-1}, P_{-2}, \ldots\}$ gives
$|\xi - \frac{18+2\sqrt{69}}{\eps - 1}|_1 = 0$, with 
$\frac{18+2\sqrt{69}}{\eps - 1} = \frac4{23}\sqrt{69}$, hence 
such $P_0 = (x,y)$ satisfy $x+y\sqrt{69} = 1+ \frac4{23}\sqrt{69}$.

Thus any point $\xi = P_0$ giving rise to a doubly infinite sequence
$(P_{j})_{j \in \Z}$ that stays inside $S_0$ modulo $\OO_K$ for all 
$j \ne 0, \pm 1$ satisfies $\xi = 1+ \frac4{23}\sqrt{69}$ and 
$\xi' = -\frac4{23}\sqrt{69}$. If we write $P_0 = (x,y)$, then
this gives $x = \frac12(\xi + \xi') = \frac12$ and
$y = \frac1{2\sqrt{69}}(\xi-\xi') = \frac4{23} + \frac1{2\sqrt{69}}
\approx 0.2341059$ as claimed.

Before we go on exploring the other possibilities, we
study the orbit of $P_0$ and compute its Euclidean minimum.
\begin{claim}
The points $P_r \equiv \eps^{-r} P_0  \bmod \OO_K$ in the orbit
of $P_0$ coincide with the $P_r$ given in Theorem \ref{T69}.
\end{claim}

This is done by induction: the case $r=0$ is clear. For the 
induction step, notice that $\eps^{-1}(x,y) = 
(\frac{25}2x - \frac{207}4y,\frac{25}2y - \frac32x)$; now
\begin{eqnarray*}
\eps^{-1}P_r & = & \Big(\frac{25}4\eps^{-r} - 18 - 
  \frac{207}{2\sqrt{69}}\eps^{-r},
  \frac{50}{23} + \frac{25}{4\sqrt{69}}\eps^{-r} - \frac34\eps^{-r}\Big) \\
 & = &(-18,2) + \Big(\Big(\frac{25}4 -\frac34\sqrt{69}\,\Big)\eps^{-r},
    \frac4{23} + \Big(-\frac34 + \frac{25}{4\sqrt{69}}\Big)\eps^{-r}\Big) \\
 & = & (-18,2) + \Big(\frac12\eps^{-r-1}, \frac4{23} + 
     \frac1{2\sqrt{69}}\eps^{-r-1}\Big) 
  \ \equiv \ P_{r+1} \bmod \OO_K 
\end{eqnarray*}
Next one computes that $\eps P_0 = (\frac{61}2,\frac72) - P_1'$
and shows, again by induction, that $\eps^r P_0 \equiv -P_r' 
\bmod \OO_K$ for all $r \ge 0$. Thus the orbit of $P_0$ under
the action of the unit group $E_K$ of $\OO_K$ is represented
modulo $\OO_K$ by the points $\{\pm P_r, \pm P_r': r \ge 0\}$.

\begin{claim} The points $P_r$ have Euclidean minimum
$$M(K,P_r) = M(K,P_0) = \frac1{46}\big(165 - 15\sqrt{69}\,).$$
\end{claim}
First we observe that the points $P_r$ have the same
Euclidean minimum since they all belong to the same orbit.
Now assume that $\eps = t+u\sqrt{m}$ has positive norm.
We want to apply Proposition \ref{Pbd} and find $\eps^{-1} = 
t-u\sqrt{m}$, hence $(\sqrt{\eps} + \frac{1}{\sqrt{\eps}}\,)^2 
 = 2t+2$ and $\mu = \sqrt{k(t+1)/2}$. In the case $m = 69$, 
we have $t = \frac{25}2$, hence $\mu/\sqrt{m} = \sqrt{k}\sqrt{27/276} < 
\frac13\sqrt{k}$. 

The orbit of  $P_0 = \frac12 + (\frac4{23} + \frac1{2\sqrt{69}})\sqrt{69}$
is $\{\pm P_r, \pm P_r': r \in \N_0\}$, so it is clearly sufficient to 
compute $M(K,P_r)$ for $r \ge 0$. 
We start with $P_0$ itself. The only $\eta \equiv P_0 \bmod \OO_K$ 
satisfying the bounds of Proposition \ref{Pbd} have the form 
$P_0 + a$ for some $a \in \Z$ or $P_0-\frac{b+\sqrt{69}}{2}$
for some odd $b \in \Z$. The minimal absolute value of the norm of 
these elements is $|N(\eta - \frac{5+\sqrt{69}}{2})| = 
\frac1{46}\big(165 - 15\sqrt{69}\,)$.

Similarly, the minimal norm for the $\eta \equiv P_1 \bmod \OO_K$
is attained at $P_1 + \frac{5-\sqrt{69}}{2}$ and again equals
$\frac1{46}\big(165 - 15\sqrt{69}\,)$.

Finally, consider the $\eta \equiv P_r \bmod \OO_K$ for some $r \ge 2$.
Then $P_r = x_r + y_r\sqrt{69}$ with $|x_r| \le 0.00081 =: \delta_0$
and $|y_r - \frac{4}{23}| < 0.0001 =: \delta_1$. The minimal absolute
value of the norm of $P_r + a$ for some $a \in \Z$ is attained
for $a = 1$, and equals $|(1+\delta_0)^2 - 69(\frac4{23}-\delta_1)^2|
\ge 1.07$; similarly, we find that $|N(P_r - \frac{b+\sqrt{69}}2)|
\ge 1.07$. 

Thus we have seen that $\inf \,\{|N(P_r - \alpha)| : 
\alpha \in \OO_K, r \in \Z\}$
is attained for $r=0$ and $\alpha = \frac{5+\sqrt{69}}{2}$, giving
$M(K,P_0) = \frac1{46}\big(165 - 15\sqrt{69}\,)$ as claimed.

\medskip

Before we go on, let us recall what we know by now: 
$K = \Q(\sqrt{69}\,)$ has first minimum
$M_1(K) = \frac{25}{23}$, and $M_1$ is isolated. Moreover, the
orbit of every $k$-exceptional point for $k = 0.875$ not congruent 
to $\pm \frac4{23}\sqrt{69} \bmod \OO_K$ has a representative
in the exceptional set $T$. Finally, if the orbit of such a point
visits $T$ exactly once, then the point is $P_0 = \frac12 + 
(\frac4{23} + \frac1{2\sqrt{69}})\sqrt{69}$, and its minimum
is $M(K,P_0) =  \frac1{46}\big(165 - 15\sqrt{69}\,)$.

\begin{claim}
Any exceptional point $Q \ne P_0$ in $T$ has Euclidean minimum
$M(K,Q) < M(K,P_0) = \frac1{46}\big(165 - 15\sqrt{69}\,)$, and 
$M_2(K) = M(P_0)$ is attained only at points in the 
orbit of $P_0$.
\end{claim}

In fact, let $Q_0 \ne P_0$ be an exceptional point in $T$
and consider the orbit $\{Q_r: r \in \Z\}$ of $Q_0$,
where the $Q_j$ are defined by $Q_j \equiv \eps^{-j}Q_0 \bmod \OO_K$.
Since $Q_0 \ne P_0$, we know that we are in one of the
following situations:
\begin{enumerate}
\item (A) and (D) hold;
\item (B) and (C) hold;
\item (B) and (D) hold.
\end{enumerate}
In each case, there exists a point $Q \ne P_0$ in $T$ whose
orbit moves into $T$ both to the right and to the left:
\begin{equation}\label{EE}
  \ldots T \lra S_2 \lra S_0 \cdots S_0 \lra S_1 \lra Q
   \lra S_2 \lra S_0 \cdots S_0 \lra S_1 \lra T \ldots 
\end{equation}
Now we prove the following lemma:

\begin{lem}\label{Lm1}
Suppose there is a $Q_0 \in T$ such that $Q_1 = \beta(Q_0-1) \in S_2$
and $Q_{m+1} = (x,y) = \beta^{m}(Q_1) \in S_1$ with $\beta$
as in (\ref{Ebe}). Then $x-y\sqrt{69} < -\frac{4}{23}\sqrt{69}$.
\end{lem}

\begin{proof}
Write $Q_n = (x_n,y_n)$ and put $\xi'_n = x_n - y_n\sqrt{69}$.
Then $\xi'_1 \approx -1.48 < -\frac{4}{23}\sqrt{69}$; now we
use induction to show that $\xi'_n < -\frac{4}{23}\sqrt{69}$
for $1 \le n \le m$. In fact, if $Q_{n+1} = \beta(x_n,y_n)$, then
$\xi_{n+1}' = (\eps \xi_n - (18 + 2\sqrt{69}\,))' 
	=  \eps' \xi'_n - 18 + 2\sqrt{69} 
        < -\eps' \frac{4}{23}\sqrt{69} - 18 + 2\sqrt{69} 
        = -\frac{4}{23}\sqrt{69}$.
\end{proof}

A similar result holds for the other direction:

\begin{lem}\label{Lm2}
Suppose there is a $Q_0 \in T$ such that $Q_{-1} = \alpha(Q_0) \in S_1$
and $Q_{-m-1} = (x,y) = \alpha^{m}(Q_{-1}) \in S_2$. Then 
$x+y\sqrt{69} > 1 + \frac{4}{23}\sqrt{69}$.
\end{lem}

\begin{proof} Similar.
\end{proof}

This shows that, in (\ref{EE}), we have
$\xi > \xi_0 = 1 + \frac{4}{23}\sqrt{69}$ and 
$\xi' < \xi_0' = -\frac{4}{23}\sqrt{69}$
for the point $Q = (x,y)$ and $\xi = x + y\sqrt{69}$,
$\xi' = x - y\sqrt{69}$. 

Put $\alpha = \xi_0 - \frac{5+\sqrt{69}}2$ and
$\alpha' = \xi_0' - \frac{5-\sqrt{69}}2$. Then
$-\alpha\alpha' = \frac1{46}\big(165 - 15\sqrt{69}\,)$,
and, since $\alpha<0$ and $\alpha'>0$, 
$0 < (\xi - \frac{5+\sqrt{69}}2\,)(\xi' - \frac{5-\sqrt{69}}2\,)
 < -\alpha\alpha'$. Thus any such point has Euclidean minimum
strictly smaller than $\frac1{46}\big(165 - 15\sqrt{69}\,)$.

\begin{claim}\label{Cl6}
The second minimum $M_2(K)$ is not isolated.
\end{claim}

This is accomplished by constructing a series of rational 
points $Q_r \in K \setminus C_2$ such that 
$\lim_{r \to \infty} M(Q_r) = M_2(K)$. 
To this end, we look for a point $Q_r \in T-1$ 
that gets mapped (multiplication by $\eps$ plus reduction 
modulo $\OO_K$) to $S_2$, stays in $S_0$ exactly $r$ times, 
and then goes to $S_1$ and back to the point in $T$ congruent 
to $Q_r \bmod \OO_K$, then $Q_r$ will satisfy 
the following equation:\footnote{For more details, see the
analogous construction of the points $R_r$ in Section \ref{S5}.}
$$ \eps^{r+4}Q_r = \eps^{r+4} + 
        (\eps^{r+3} + \ldots + \eps + 1)(18+2\sqrt{69}\,) + Q_r.$$
This gives 
$$Q_r = 1 + \frac{4}{23}\sqrt{69} + \frac1{\eps^{r+4}-1}.$$
Here's a short table with explicit coordinates for small values of $r$:
$$ \begin{array}{|r|c|cl|}\hline
\rsp  r   &   Q_r  &   \multicolumn{2}{c|}{M(Q_r)} \\ \hline
\rsp -1 &  \frac{1}{2} + \frac{97}{414}\sqrt{69}       &  
       \frac{541}{621}     &\approx 0.871175523 \\
\rsp  0 &  \frac{1}{2} + \frac{70}{299}\sqrt{69}       &  
       \frac{13651}{15548} &\approx 0.877990738  \\
\rsp  1 &  \frac{1}{2} + \frac{2423}{10350}\sqrt{69}   &  
       \frac{340876}{388125} &\approx 0.878263446  \\
\rsp  2 &  \frac{1}{2} + \frac{6989}{29854}\sqrt{69}   &  
       \frac{8508391}{9687623} &\approx  0.878274371  \\
\rsp  3 &  \frac{1}{2} + \frac{30239}{129168}\sqrt{69} &  
       \frac{212369041}{241802496} &\approx 0.878274809  \\
\rsp  4 &  \frac{1}{2} + \frac{174445}{745154}\sqrt{69} &  
       \frac{5300717776}{6035374823} &\approx 0.878274826 \\ 
\hline \end{array}$$
We claim that $M(Q_r)$ tends to 
$M_2(K) = \frac1{46}\left(165-15\sqrt{69}\,\right) \approx 0.87827$ 
as $r \lra \infty$. Applying Proposition \ref{Pbd} shows that,
for given $r \ge 0$, the Euclidean minimum of $Q_r$ is attained 
at $Q_r - \frac{5+\sqrt{69}}{2}$. Writing $n = r+4$ and 
$Q_r - \frac{5+\sqrt{69}}{2} = (\xi,\xi')$ we have 
\begin{eqnarray*}
  \xi & = & -\frac32 - \frac{15}{46} \sqrt{69} + \frac1{\eps^n-1}, \\
 \xi' & = & -\frac32 + \frac{15}{46} \sqrt{69} + \frac1{\eps^{-n}-1} 
      \ = \ -\frac52 + \frac{15}{46} \sqrt{69} - \frac1{\eps^n-1}, 
\end{eqnarray*}
and now we find
$$\Big|N\Big(Q_r - \frac{5+\sqrt{69}}{2}\Big)\Big| = -\xi\xi'
  = \frac{165-15\sqrt{69}}{46} - 
     \frac1{\eps^n-1}\Big(-1+\frac{15}{23}\sqrt{69}\,\Big).$$
Since the ``error term'' $\frac1{\eps^n-1}(-1+\frac{15}{23}\sqrt{69}\,)$
is positive and tends to $0$ as $n \lra \infty$, Claim \ref{Cl6} follows,
and Theorem \ref{T69} is proved.

\section{Weighted norms in $\Q(\sqrt{69}\,)$}\label{S5}

Now we study the weighted norm $f_{\fp,c}$ defined by 
$\fp = (23,\sqrt{69}\,)$. We claim 

\begin{thm}\label{T69w}
Let $R = \OO_K$ be the ring of integers in $K = \Q(\sqrt{69}\,)$,
and let $\fp = (23,\sqrt{69}\,)$ be the prime ideal above $23$. 
Then the Euclidean window of $f = f_{\fp,c}$ is $w(\fp) = 
(25,\infty)$; the Euclidean minimum is 
$$M_1(\OO_K,f_{\fp,c}) = \max\Big\{\frac{25}{c},
\frac1{23}(-600+75\sqrt{69}\,)\Big\}$$
for all $c \in w(\fp)$, and $M_1$ is isolated exactly
when $c \in [23, \frac{23}{15}(8+\sqrt{69}\,))$.
\end{thm}

Using the method described in \cite{Cl1}, with some modifications
described in the next section, we can cover the fundamental
domain of $\OO_K$ with a bound of $k = 0.99$ except for a
set surrounding $(0,0)$ that contains no exceptional point, and
$\pm S_1 \cup \pm S_2 \cup \pm S_2'$, where
$$\begin{array}{rcrcl}
  S_1 & = & [-0.0084, 0.0084] & \times & [0.1739, 0.175] \\
  S_2 & = & [0.2086, 0.2087]  & \times & [0.19903, 0.19904] \\
 S_2' & = & [0.2086, 0.2087]  & \times & [-0.19904, -0.19903] \\
\end{array}$$
Transforming by units, we find
$$ \begin{array}{rclrcl}
 \eps S_1 - (18+2\sqrt{69}\,) & \subs & S_1 \cup S_2, & 
 \oeps S_1 + (18-2\sqrt{69}\,) & \subs & S_1 \cup (-S_2'), \\
 \eps S_2 - (23+3\sqrt{69}\,) & \subs & \phantom{-}S_2', &
 \oeps S_2 + (18-2\sqrt{69}\,) & \subs & S_1, \\ 
 \eps S_2'+ (18+2\sqrt{69}\,) & \subs & -S_1, &
 \oeps S_2' - (23-3\sqrt{69}\,) & \subs & S_2. \end{array}$$

\begin{claim}\label{C51}
If $P$ is an exceptional point that stays inside $S_1$ 
under repeated transformations by $\eps$ and $\eps^{-1}$, 
then $P = (0,\frac4{23})$ has Euclidean minimum 
$M(P,f_{\fp,c}) = \frac{25}{23c}$. 
\end{claim}

This is easy to see. Again, this enables us to reduce everything
to exceptional points $P \in S_2$, and for the orbit $(P_j)$ of 
such $P$ (here $P_{j+1}$ is the image of $P_j$ under multiplication
by $\eps$ plus reduction modulo $\OO_K$) there are the following 
possibilities:

\begin{enumerate}
\item[(a)] $P_j \in -S_1$ and $P_{-j} \in S_1$ for all $j \ge 2$;
\item[(b)] there exist $m \ne n$ such that $P_m, P_n \in S_2$.
\end{enumerate}

\begin{claim}
If $P_0 \in S_2$ is an exceptional point with property (a), then 
$$  P_0   = \ts (\frac{-115+15\sqrt{69}}{46} , 
     \frac{-5+\sqrt{69}}{2\sqrt{69}}\,)
      \approx  (0.20868169, 0.19903536). $$
\end{claim}

For a proof, suppose that $P_0$ is a point in $S_2$ with property (a). 
Then $P_1 = - \eps P_0 + (23 + 3 \sqrt{69}\,) \in -S_2'$, 
and $P_2 = \eps P_1 - (18 + 2\sqrt{69}\,)$ is a point 
whose transforms by powers of $\eps$ stay inside $S_1$. 
By Proposition \ref{PEP}, this implies that 
$|P_2 - \frac4{23}\sqrt{69}|_1 = 0$, and going back to 
$P_0$ we find that $|P_0 - (-5 + \frac{19}{23}\sqrt{69}\,)|_1 = 0$.

Similarly, any exceptional point $\xi \in S_2$ whose 
transforms by powers of $\oeps$ stay inside $S_1$ satisfies
$|\xi + \frac4{23}\sqrt{69}|_2 = 0$. Thus any point satisfying
(a) has $x$-coordinate $(\xi+\xi')/2 = \frac{-115+15\sqrt{69}}{46}$
and $y$-coordinate $(\xi-\xi')/2\sqrt{69} = 
\frac{-5+\sqrt{69}}{2\sqrt{69}}$ as claimed. 

Note that there is no obvious definition of a ``Euclidean minimum''
of $P_0$ with respect to weighted norms $f_{\fp,c}$, since
$f_{\fp,c}$ is a continuous function on $K$ (with respect to the
topology inherited from the embedding $K \lra \R^2$) if and
only if $c = \fp$, that is, if and only if $f_{\fp,c}$ is the
absolute value of the usual norm. Thus we cannot extend 
$f_{\fp,c}$ by continuity to $\R^2$. On the other hand,
we can put 
$$\bM(P,f_{\fp,c}) = \sup \ \{M(P_r,f_{\fp,c}): P_r \in K,
	\ \lim P_r = P\},$$
that is, define the minimum at a point $P \in K$ as the
supremum of the minima at $P_r \in K$ over all sequences $(P_r)$
converging to $P$ in the topology mentioned above.  
If $P \in K$, then clearly $\bM(P,f_{\fp,c}) \ge M(P,f_{\fp,c})$,
as the constant series $P_r = P$ shows. We don't know an
example where this last inequality is strict.

\begin{claim}
We have $\bM(P_0) \le \kappa_0 = \frac1{23}(-600+75\sqrt{69}\,)$ 
for all $c \ge 23$. Moreover, any $K$-rational exceptional point 
with property (b) has minimum strictly smaller than $\kappa_0$.
In particular, we have $M_1(K) = \frac{25}{23c}$ for all
$c \in [23, \frac{25}{23}(24+3\sqrt{69}\,)]$, 
and $M_1$ is isolated for these values of $c$ unless possibly
when $c = \frac{25}{23}(24+3\sqrt{69}\,)$. 
\end{claim}

We start by observing that 
$$ \begin{array}{rccl} \smallskip
|N(P_0 - 2)| & = & \ts  \frac{94-10\sqrt{69}}{23} 
      & \approx  0.47538092916, \quad \text{and} \\ 
\ts |N(P_0 - \frac12(5 + \sqrt{69}\,))| & = &
           \frac{-600+75\sqrt{69}}{23} & \approx  0.99986042255.
\end{array}$$

Using the same technique as in Lemma \ref{Lm1} and \ref{Lm2}
we can show that the $K$-rational points in $S_2$ that satisfy
condition (b) have minimum strictly smaller than $\kappa_0$;
observe that the difference $\eta_1 - \eta_2$ for 
$\eta_1 = \frac12(5 + \sqrt{69}\,)$ and $\eta_2 = 2$ is
not divisible by $\fp$, hence we have $f_{\fp,c}(P_0 - \eta_j)
\le |N(P_0 - \eta_j)|$ for $j = 1$ or $j = 2$. 
Since any sequence of $K$-rational points $P_r$ converging 
to $P_0$ eventually stays inside $S_2$ this also proves that
$M_1(\OO_K,f_{\fp,c}) = \frac{25}{23c}$ as long as 
$\frac{25}{23c} \ge \kappa_0$; but the last inequality
holds for all $c \le \frac{23}{15}(8+\sqrt{69}\,) \approx 25.0034899$.
It also shows that the minimum is isolated for these values
unless possibly when $c = \frac{23}{15}(8+\sqrt{69}\,)$.

\begin{claim}\label{CMP}
We have $\bM(P_0,f_{\fp,c}) = \kappa_0 = \frac1{23}(-600+75\sqrt{69}\,)$ 
for all $c > 23$, and $\bM(P_0,f_{\fp,c}) = M(P) 
= \frac{94- 10\sqrt{69}}{23}$ for $c = 23$. 
\end{claim}

In order to show that $\kappa_0$ is a lower bound for $M(P_0)$
for $c > 23$, we construct a series of 
$K$-rational points converging to $P_0$ whose minima converge 
to $\kappa_0$. We do this in the following way: assume that 
$R_r \in S_2$ gets mapped to $S_2'$, stays in $-S_1$ exactly 
$r-2$ times and then gets mapped to the point $-R_r \in -S_2$. 
Then $\eps R_r - (23+3\sqrt{69}\,) \in S_2'$,
$\eps^2 R_r - \eps(23+3\sqrt{69}\,) + (18+2\sqrt{69}\,) \in -S_1$,
\ldots, $\eps^r R_r - \eps^{r-1}(23+3\sqrt{69}\,) + (18+2\sqrt{69}\,)
(1 + \eps + \ldots + \eps^{r-2}) \in -S_1$ and finally
\begin{equation*}
 (\eps^{r+1}+1)R_r  =  \eps^r(23+3\sqrt{69}\,) - 
           (18+2\sqrt{69}\,) \frac{\eps^r-1}{\eps-1} 
\end{equation*}
Now we use $\frac{\eps^r-1}{\eps-1} = \frac{\eps^{r+1}-1}{\eps-1} - \eps^r$
to find
\begin{eqnarray*}
(\eps^{r+1}+1)R_r & = & \eps^r(41+5\sqrt{69}\,) - 
           (18+2\sqrt{69}\,) \frac{\eps^{r+1}-1}{\eps-1} \\
 & = & \eps^{r+1}(-5+\sqrt{69}\,) - 
           (18+2\sqrt{69}\,) \frac{\eps^{r+1}-1}{\eps-1}.
\end{eqnarray*}
Dividing through by $\eps^{r+1}+1$ and simplifying we get
$$ R_r  =  -5 + \frac{19}{23}\sqrt{69} + 
           \frac1{\eps^{r+1}+1}\Big(5 - \frac{15}{23}\sqrt{69}\,\Big).$$

The explicit coordinates for the first few points are given
in the following table:
$$ \begin{array}{|r|c|cl|cl|}\hline
\rsp  r   &   R_r  & 
      \multicolumn{2}{c|}{|N(R_r-\frac12(5 + \sqrt{69}\,))|} &
      \multicolumn{2}{c|}{|N(R_r- 2)|} \\ \hline
\rsp 1 &  \frac{1}{5} +  \frac{1}{5}\sqrt{69}  
    &  \frac{23}{25} &= 0.92 & \frac{12}{25} & = 0.48 \\
\rsp 2 &  \frac{5}{24} +  \frac{43}{216} \sqrt{69} 
    &  \frac{3875}{3888} &\approx 0.996656378 
    &  \frac{1849}{3888} &\approx 0.475565843  \\
\rsp 3 &  \frac{130}{623} +  \frac{124}{623}\sqrt{69}  
    &  \frac{388025}{388129} &\approx 0.999732047
    &  \frac{184512}{388129} &\approx 0.475388337  \\
\rsp 4 &  \frac{125}{599} +  \frac{1073}{5391}\sqrt{69}  
    &  \frac{9686225}{9687627}  &\approx 0.999855279 
    &  \frac{4605316}{9687627}  &\approx 0.475381225  \\
\rsp 5 &  \frac{649}{3110} +  \frac{619}{3110}\sqrt{69}  
    &  \frac{2417687}{2418025} &\approx 0.999860216 
    &  \frac{1149483}{2418025} &\approx 0.475380941  \\
\rsp 6 &  \frac{3120}{14951} +  \frac{26782}{134559}\sqrt{69}  
    &  \frac{6034532375}{6035374827}  &\approx 0.999860414
    &  \frac{2935561516}{6035374827}  &\approx 0.475380929\\
\hline \end{array}$$

\begin{claim}
The Euclidean minimum of $R_r$ ($r \ge 2$) with respect to 
$f_{\fp,c}$ is attained at $R_r - 2$ or 
$R_r-\frac12(5 + \sqrt{69}\,)$.
\end{claim}

In fact, applying Proposition \ref{Pbd} to $R_r$ one checks
that the two smallest values of $|N(R_r - \eta)|$ occur for
$\eta_1 = 2$ or $\eta_2 = \frac12(5 + \sqrt{69}\,)$; one also
verifies that $|N(R_r - 2)| \approx 0.47$ and 
$|N(R_r-\frac12(5 + \sqrt{69}\,))| \approx 0.99$.
Since the denominator of $R_r-\eta$ is not divisible by $\fp$
for any $\eta \in \OO_K$ (it divides $\eps^{r+1}+1 \equiv 2 \bmod \fp$), 
and since $\eta_1 - \eta_2$ is an integer not divisible by $\fp$,
our claim follows.

Where the minimum with respect to $f_{\fp,c}$ is attained 
depends on whether the numerator of $R_r-2$ is divisible by 
$\fp$ or not: if it isn't, then the Euclidean minimum 
is attained there, and we have $M(P,f_{\fp,c}) = |N(R_r-2)| < \frac12$.
If this numerator, however, {\em is} divisible by $\fp$,
then $f_{\fp,c}(R_r - 2)$ can be made as large as we please
by adding weight to $\fp$, and in this case the minimum
is attained at $R_r-\frac12(5 + \sqrt{69}\,)$ for large
values of $c$. 

\begin{claim}\label{C6}
The numerator of $R_r-2$ is divisible by $\fp$ if and only
if $r \equiv 10 \bmod 23$. In this case, it is even divisible
by $(23) = \fp^2$. 
\end{claim}

Let us compute $R_r \bmod \fp$. Since $\eps \equiv 1 \bmod \fp$,
we find $\frac{\eps^{r+1}-1}{\eps-1} = 1 + \eps + \ldots + \eps^r 
\equiv r+1 \bmod \fp$, hence
$2R_r = \eps^r(23+3\sqrt{69}\,) - (18+2\sqrt{69}\,) \frac{\eps^r-1}{\eps-1} 
 \equiv 5r \bmod \fp$, and therefore $R_r - 2 \equiv 0 \bmod \fp$
if and only if $5r \equiv 4 \bmod 23$, which in turn is
equivalent to $r \equiv 10 \bmod 23$.

The second part of the claim follows by observing
$\eps^s \equiv (1 + 13\sqrt{69}\,) \equiv 1+13s\sqrt{69} \bmod 23$, 
in particular $\eps^{23m+10} \equiv 1+13\sqrt{69} \bmod 23$ and
$\frac{\eps^r-1}{\eps-1} = \eps^{r-1} + \ldots + \eps+1 \equiv
r+1 + 13 \frac{r(r+1)}2 \sqrt{69} \bmod 23$. 

With a little more effort we can show much more, namely that
there is a subsequence of $R_r-2$ with numerators divisible
by an arbitrarily large power of $\fp$. In fact, the 
numerator of $R_r-2$ will be divisible by $\fp^k$ if and
only if $T_r = 23(\eps^{r+1}+1)(R_r-2) \equiv 0 \bmod \fp^{k+2}$,
and here $T_r$ is an algebraic integer. An elementary 
calculation shows that the last congruence is equivalent to
\begin{equation} 
  \eps^{r+1} \equiv - \frac{47+5\sqrt{69}}{22}=: \alpha \bmod \fp^{k+2}.
\end {equation} 
This will hold for arbitrarily large $k$ if and only if there
is a $23$-adic integer $s = r+1$ such that 
\begin{equation}\label{E23}
\eps^s = \alpha
\end {equation} 
holds in $K_\fp = \Q_{23}(\sqrt{69}\,)$. Since both sides are 
congruent $1 \bmod \fp$, we can take the $\pi$-adic logarithm
(with $\pi = \frac{23+3\sqrt{69}}{2}$)
and get $s = \frac{\log_\pi \alpha}{\log_\pi \eps}$ as
an equation in $K_\fp$, and (\ref{E23}) holds if we can show
that $s$ is in $\Z_{23}$. To this end,\footnote{We thank 
(in chronological order) Hendrik Lenstra, Gerhard Niklasch 
and David Kohel for this argument.} 
let $\sigma$ denote the non-trivial
automorphism of $K_\fp/\Q_{23}$. Since $\log_\pi$ is Galois-equivariant,
and since $\eps^{1+\sigma} = \alpha^{1+\sigma} = 1$, we get
$$s^\sigma = \frac{\log_\pi \alpha^\sigma}{\log_\pi \eps^\sigma} 
  = \frac{-\log_\pi \alpha}{-\log_\pi \eps} = s.$$
Thus $s \in \Q_{23}$, and since it is a $\pi$-adic unit, 
$s \in \Z_{23}$ as desired. We remark that 
$s = 11 + 13 \cdot 23 + 15 \cdot 23^2 + 5 \cdot 23^3 + 
	3 \cdot 23^4 + \ldots$.

This proves Claim \ref{CMP} and completes the proof of 
Theorem \ref{T69w}.

\section{Weighted norms in cubic number fields}

Using the idea of Clark (see \cite{Cl1,Cl2,Hai,Nik}; it 
actually first appears in Lenstra \cite[p. 35]{Len}),
we modified the programs described in \cite{CL} slightly in 
order to examine weighted norms in cubic fields.
Many of the results in this section have been obtained by 
the first author in \cite{Cav}; see Table \ref{TR} for
the results obtained so far. 


\begin{table}[!ht]
\begin{center}
\caption{}\label{TR} \medskip
\begin{tabular}{|r|c|c|c|c|}\hline
\rsp $\disc K$ & $M_1(K)$ & $M_2(K)$ & $N\fp$ & w$(\fp)$ \\ \hline \hline
\rsp $-367$ & $1$       &  $9/13$ & $13$  
                        & $(13, 279/8)$  \\ \hline
\rsp $-351$ & $1$       &  $9/11$ & $11$   
                        & $(11, \infty)$        \\ \hline
\rsp $-327$ & $101/99$  & $<0.9$ & $11$      
                        & $(101/9, \infty)$  \\ \hline
\rsp $-199$ & $1 $      & $< 0.47$ & $7$  
                        & $(7, \infty)$         \\ \hline
\rsp $ 985$ & $1 $      & $5/11$ & $5$    
                        & $(5,\infty)$          \\ \hline
\rsp $1345$ & $7/5$     & $< 0.4$  & $5$    
                        & $(7, \infty)$         \\ \hline
\rsp $1825$ & $7/5$     & $< 0.5$  & $5$      
                        & $(7, \infty)$         \\ \hline
\rsp $1929$ & $1$       & $3/7$ & $7$       
                        & $(7, \infty)$         \\ \hline
\rsp $1937$ & $1$       & $5/9$ & $3$       
                        & $(3, \infty)$         \\ \hline
\rsp $2777$ & $5/3$     & $17/19$  & $3$ 
                        & $\varnothing$         \\ \hline
\rsp $2836$ & $7/4$     & $7/8$    & $2$
			& $(\sqrt{7}, \infty)$	\\ \hline
\rsp $2857$ & $8/5$     & $< 0.5$  & $5$    
                        & $(8, \infty)$         \\ \hline
\rsp $3305$ & $13/9$    & $37/45$ & $3$       
                        & $(\sqrt{13}, 5)$      \\ \hline
\rsp $3889$ & $13/7$    & 1  & $7$ 
                        & $(13,\infty)$         \\ \hline
\rsp $4193$ & $7/5$     & $< 0.65$ & $5$    
                        & $(7, \infty)$         \\ \hline
\rsp $4345$ & $7/5$     & $11/13$ & $5$     
                        & $(7, \infty)$         \\ \hline
\rsp $4360$ & $41/35$     & $7/10$ & $7$ 
                        & $(41/5,\infty)$       \\ \hline
\rsp $5089$ & $17/11$       & $7/11$  & $11$  
                        & $(17,\infty)$         \\ \hline
\rsp $5281$ & $1$           & $<0.6$  & $5$ 
                        &  $(5,\infty)$    \\  \hline
\rsp $5297$ & $21/11$ & $23/33$ & $11$ 
                        &  $(21,\infty)$ \\ \hline
\rsp $5329$ & $9/8$  &  $63/73$     &  $2^3$ 
                        &  $(9,73)$ \\ \hline 
\rsp $5369$ & $21/19$    &  $17/19$ &  $19$ 
                        &  $(21,\infty)$   \\  \hline
\rsp $5521$ & $23/7$  & $8/7$ & $7$ 
                        &  $(23,\infty)$ \\ \hline
\rsp $7273$ & $973/601$ & $729/601$ & $601$
                        &  $(973, \infty)$ \\ \hline
\rsp $7465$ & $1$       & $< 0.8$ & $5$ 
                        &  $(5,\infty)$ \\ \hline
\rsp $7481$ & $1$       & $< 0.7$ & $5$ 
                        & $(5,\infty)$ \\ \hline
\end{tabular}
\end{center}
\end{table}

The idea is simple. Assume that $K$ is a number field
with class number $1$ such that $M = M_1(K) \ge 1$ and $M_2(K) < 1$; 
assume that $\# C_1(K)$ is finite and write 
the points  $\xi \in C_1(K)$ ($1 \le i \le t$) 
in the form $\xi_i = \alpha_i/\beta_i$, where $(\alpha_i,\beta_i) = 1$. 
Assume moreover that there is a prime ideal $\fp$ such that 
$\fp \mid \beta_i$ for all $i$. 

Now consider the weighted norm $f_{\fp,c}$; by making $c$ big enough
we can certainly arrange that $f_{\fp,c}(\xi_i) < 1$ for all $i \le t$:
in fact, if $\fp^m \parallel \gcd(\beta_1, \ldots, \beta_t)$, then 
$f_{\fp,c}(\xi_i) \le M (N\fp)^m c^{-m}$; thus we only need to choose 
$c > N\fp \sqrt[m\,]{M}$ (actually this shows that $w(\fp) \subseteq 
(N\fp \sqrt[m\,]{M},\infty)$).

In order to guarantee that, for every $\xi \in K$, there
exists a $\gamma \in \OO_K$ such that $f_{\fp,c}(\xi-\gamma) < 1$,
we will look for $\gamma_1, \gamma_2 \in \OO_K$ such that
$|N_{K/\Q}(\xi-\gamma_i)| < 1$ for $i=1, 2$ and 
$\fp \nmid (\gamma_1 - \gamma_2)$; then at least one of 
the $\xi - \gamma_i$, say $\xi - \gamma_1$, has numerator not 
divisible by $\fp$, and this implies that 
$f_{\fp,c}(\xi-\gamma_1) \le |N(\xi-\gamma_1)| < 1$.

By modifying the programs described in \cite{CL} slightly
we can use them to find new examples of cubic fields that are
not norm-Euclidean but Euclidean with respect to some weighted
norm. We represented prime ideals of the maximal order 
$\OO_K = \Z \oplus \alpha\Z \oplus \beta\Z$ in the form 
$\fp = (p, \alpha + a)$, $(p,\beta + a\alpha + b)$ or $(p)$ 
according as $\fp$ has degree $1$, $2$ or $3$. Testing the
divisibility of an integer of $\OO_K$ by $\fp$ then can be
done using only rational arithmetic.

Let us call $\xi \in K$ covered if there exist
$\gamma_1, \gamma_2 \in \OO_K$ such that
$|N_{K/\Q}(\xi-\gamma_i)| < 1$ and $\fp \nmid (\gamma_1 - \gamma_2)$; 
if $\xi$ is covered, then so is $\eps \xi$ for any unit 
$\eps \in \OO_K^\times$ (this allows us to use the program
E--3 of \cite{CL}).

We first consider the field $K$ generated by a root $\alpha$
of $x^3 + x^2 - 6x - 1$; we have $\disc K = 985$, and the
only point with minimum $\ge 1$ is 
$\xi_1 = \frac{3\alpha-\alpha^2}{\alpha-1} = \frac{2-\alpha+2\alpha^2}{5}.$
The ideal $\fp = (\alpha - 1)$ occurring in the denominator is a 
prime ideal of norm $5$. Our programs cover a fundamental domain
of $K$ except for the possible exceptional points $\xi = 0$
and $\xi = \xi_1$. Thus $f_{\fp,c}$ is a Euclidean function
for every $c > N\fp = 5$, i.e. $w(\fp) = (5,\infty)$.

Now let $K$ be the field with $\disc K = 1937$ generated by a 
root $\alpha$ of $x^3 + x^2 - 8x + 1$. It has Euclidean
minimum $M(K) = 1$ attained at $\frac{4+4\alpha^2}{9}$; in fact 
$|N(\xi_1)| = 1$ for $\xi_1 = \frac19(-14+9\alpha+4\alpha^2)$, 
and the prime ideal factorization of $\xi_1$ is
$(\xi_1) = (3,\alpha^2+1)(3,\alpha+1)^{-2}$. Our programs cover
a fundamental domain of $K$ except for the possible exceptional 
points $\xi_0 = 0$, $\xi = \xi_1$ and $\xi = \frac13(1+\alpha^2)$.
This last point has Euclidean minimum 
$\frac13 = |N(\frac13(1-3\alpha+\alpha^2))|$ with respect
to the usual norm, and since $(1-3\alpha+\alpha^2)/3 = \fp^{-1}$,
adding weight to $\fp$ does not increase its minimum.

Our third example is the cubic field $K$ with discriminant
$\disc K = 3305$, generated by a root  $\alpha$ of
$x^3 - x^2 - 10x - 3$. It has minimum $M_1 = \frac{13}{9}$
attained at $\frac19(1 - 2\alpha - 4\alpha^2)$, with
$|N(\xi_1)| = \frac{13}{9}$ for 
$\xi_1 = \frac19(-71 + 52\alpha + 32\alpha^2)$. Its prime ideal
factorization is $(\xi_1) = (13,\alpha-1)(3,\alpha)^{-2}$; we thus
add weight $c > \sqrt{13}$ to $\fp = (3,\alpha)$, and we can 
cover a fundamental domain of $K$ except for the possible 
exceptional points $\xi_0 = 0$, $\xi = \xi_1$ and 
$\xi = \frac15(2-\alpha+2\alpha^2)$. 
Now $M(\xi) = |N(\xi_2)| = \frac35$, where 
$\xi_2 = \frac15(-3+4\alpha+2\alpha^2)$ has the prime ideal
factorization $(\xi_2) = \fp(5,\alpha+2)^{-1}$. Thus the
weighted prime ideal occurs in the numerator of $\xi_2$, and we
have  $f_{\fp,c}(\xi_2) < 1$ if and only if $c < 5$; since
$|N(\xi)| \ge 1$ for all $\xi \equiv \xi_2 \bmod \OO_K$,
this implies that $w(\fp) = (\sqrt{13},5)$.

Finally, consider the cubic field $K$ with discriminant
$\disc K = 3889$. Its first minimum is attained at 
$\xi_1 = \frac17(3-\alpha - 3\alpha^2)$, and its denominator
is the prime ideal $\fp$ that divides the denominator of
$\xi_2 = \frac17(2-3\alpha - 2\alpha^2)$, where the second 
minimum $M_2(K) = 1$ is attained (something similar happens 
for $\disc K = 5521$ and $\disc K = 7273$, where $M_2(K) > 1$;
in these cases, we have to verify that $M_3(K) < 1$). Here
we find the possible exceptional points $\xi = 0$, $\xi_1$, 
$\xi_2$, as well as 
$\eta_1 = \frac17(1 - \alpha - 2\alpha^2)$, 
$\eta_2 = \frac17(2 - 2\alpha + 3\alpha^2)$ and
$\eta_3 = \frac17(3  -3\alpha + \alpha^2)$.
Since their denominator is the prime ideal $(7, 2+\alpha)$,
their Euclidean minimum is $\frac17$ both for the usual as well
as for the weighted norm.

Some of our examples of cubic fields that are Euclidean with
respect to some weighted norm were found independently by 
Amin Coja-Oghlan; see his forthcoming thesis \cite{CO}.

\section{Norm-Euclidean cubic fields}
We take this opportunity to report on recent computations
concerning norm-Euclidean cubic fields. Calculations for the 
totally real cubic fields up to $\disc K \le 13,000$ have 
produced the following results:

\begin{center}
\begin{tabular}{|r|r|r|r|r|}\hline
$\disc K \qquad $ & E & N &  $\Sigma$ \\ \hline \hline
 $   0 < d \le \phantom{0} 1000$ & 26 &     1 &        27 \\
 $1000 < d \le \phantom{0} 2000$ & 29 &     5 &        34 \\
 $2000 < d \le \phantom{0} 3000$ & 31 &     4 &        35 \\
 $3000 < d \le \phantom{0} 4000$ & 36 &     6 &        42 \\
 $4000 < d \le \phantom{0} 5000$ & 28 &     7 &        35 \\
 $5000 < d \le \phantom{0} 6000$ & 35 &     7 &        42 \\
 $6000 < d \le \phantom{0} 7000$ & 30 &     8 &        38 \\
 $7000 < d \le \phantom{0} 8000$ & 37 &    10 &        47 \\
 $8000 < d \le \phantom{0} 9000$ & 30 &    11 &        41 \\ 
 $9000 < d \le 10000$            & 29 &    10 &        39 \\ 
$10000 < d \le 11000$            & 34 &     9 &        43 \\
$11000 < d \le 12000$            & 37 &    16 &        53 \\
$12000 < d \le 13000$            & 31 &     6 &        37 \\ \hline
    $\Sigma \qquad $             &413 &   100 &       513 \\ \hline
\end{tabular}
\end{center}

\medskip

The columns $E$ and $N$ display the number of norm-Euclidean and
not norm-Euclidean number fields of fields with discriminants
in the indicated intervals. 


We also have to correct the entries for the fields with
discriminant $3969$ in our tables in \cite{CL}:
the field $K_1$ generated by a root of $x^3-21x-28$ has
$M_1(K_1) = 4/3$, $M_2(K_1) = 31/24$ and $M_3(K_1) = 1$, and the
field $K_2$ generated by $x^3-21x-35$ has $M_1(K_2) = 7/3$
and $M_2(K_2) = 125/63$.

For complex cubic fields, calculations by R. Qu\^eme indicated 
that the fields with $\disc K = -999$ and $\disc K = -1055$ 
are not norm-Euclidean, and we could meanwhile verify that 
$M(K) \ge 294557/272112$ for $\disc K = -999$ and
$M(K) \ge 1483/1370$ for $\disc K = -1055$, and that
there are no norm-Euclidean number fields with 
$-876 > \disc K \ge -1600$, suggesting the following

\medskip
\noindent {\bf Conjecture.} There are exactly $58$ norm-Euclidean 
complex cubic fields, and their discriminants are
 $-23$,  $-31$,  $-44$,  $-59$,  $-76$,  $-83$,  $-87$, $-104$,
$-107$, $-108$, $-116$, $-135$, $-139$, $-140$, $-152$, $-172$, 
$-175$, $-200$, $-204$, $-211$, $-212$, $-216$, $-231$, $-239$, 
$-243$, $-244$, $-247$, $-255$, $-268$, $-300$, $-324$, $-356$, 
$-379$, $-411$, $-419$, $-424$, $-431$, $-440$, $-451$, $-460$, 
$-472$, $-484$, $-492$, $-499$, $-503$, $-515$, $-516$, $-519$, 
$-543$, $-628$, $-652$, $-687$, $-696$, $-728$, $-744$, $-771$, 
$-815$, $-876$.

\medskip

Note that, by a result of Cassels \cite{Cas52}, there are
only finitely many norm-Euclidean complex cubic number
fields $K$, and in fact their discriminant is bounded by
$|\disc K| < 170\,520$.

\begin{center}
\begin{tabular}{|r|r|r|r|}\hline
$d = |\disc K| \qquad $ & E & N  & $\Sigma$ \\ \hline \hline
 $  0 < d \le \phantom{0}200$ & 18    &    1   &   19 \\
 $200 < d \le \phantom{0}400$ & 15    &    9   &   24 \\
 $400 < d \le \phantom{0}600$ & 16    &   10   &   26 \\
 $600 < d \le \phantom{0}800$ &  7    &   20   &   27 \\
 $800 < d \le 1000$ 	      &  2    &   29   &   31 \\
$1000 < d \le 1200$ 	      &  0    &   29   &   29 \\
$1200 < d \le 1400$ 	      &  0    &   35   &   35 \\  
$1400 < d \le 1600$ 	      &  0    &   27   &   27 \\ \hline 
   $\Sigma  \qquad$ 	      & 58    &  160   &  218 \\ \hline
\end{tabular}
\end{center}

In the real case, the situation is not so clear. The numerical
data suggest that the proportion of norm-Euclidean fields
is decreasing with $\disc K$, but they do not yet support the
conjecture that the norm-Euclidean real cubic number fields 
have density $0$ among the real cubic fields with class number $1$.

\section{Some Open Problems}

In this last section we would like to mention several
open problems concerning the Euclidean algorithm with 
respect to weighted norms. One of the most studied 
questions is of course whether $\Z[\sqrt{14}\,]$ is
Euclidean with respect to some $f_{\fp,c}$, where
$\fp = (2,\sqrt{14}\,)$. Is it true, in particular,
that $w(\fp) = (\sqrt{5},\sqrt{7}\,)$ in this case?

More generally: assume that $K$ is a number field with 
unit rank $\ge 1$. Is $w(\fp)$ always an open subset of 
$(1,\infty) \subset \R$ for every prime ideal $\fp$ in $\OO_K$?
If this were the case, then there would also exist
number fields such that $f_{\fp,c}$ is a Euclidean function 
for some $c < N\fp$ since there do exist number fields
with $w(\fp) \supseteq [p,\infty)$ for suitable primes
(take norm-Euclidean fields, for example). 

A related question is whether $M(f_{\fp,c})$ is a continuous 
function of $c$ on $[N\fp,\infty)$ for number fields with unit 
rank $\ge 1$.

The cubic field with discriminant $\disc K = -335$ has $M_1(K) = 1$;
the minimum is attained at points that have different prime ideals
above $5$ in their denominator. Calculations have not yet
confirmed that $\OO_K$ is Euclidean with respect to a norm
that is weighted at two different prime ideals. Similar remarks
apply to algorithms with respect to functions that are not
multiplicative: instead of giving weight $c$ to a prime ideal $\fp$,
one could look at functions with $f(\fp) = N\fp$ and $f(\fp^2) = c$
for some $c \ge N\fp^2$. This idea is applicable whenever the
denominators of the exceptional points are divisible by the
square of a prime ideal, e.g. for $\Z[\sqrt{14}\,]$.

\end{document}